\documentclass[11pt]{amsart}
\usepackage{amscd,amssymb,url,colonequals,enumitem,xy}
\xyoption{all}
\newtheorem{thm}{Theorem}[section]
\newtheorem{prop}[thm]{Proposition}
\newtheorem{lem}[thm]{Lemma}
\newtheorem{cor}[thm]{Corollary}

\theoremstyle{definition}
\newtheorem{defn}[thm]{Definition}
\newtheorem{exa}[thm]{Example}

\theoremstyle{remark}
\newtheorem{rem}[thm]{Remark}

\numberwithin{equation}{section}

\newcommand{\op}{\operatorname}

\newcommand{\Ker}{\operatorname{Ker}}
\newcommand{\Gr}{\operatorname{Gr}}
\newcommand{\Char}{\op{char}}
\newcommand{\Lie}{\op{Lie}}

\newcommand{\Span}{\op{Span}}
\newcommand{\Rep}{\op{Rep}}
\newcommand{\rank}{\op{rank}}
\newcommand{\Split}{\op{Split}_k}
\newcommand{\ed}{\op{ed}}
\newcommand{\Id}{\op{Id}}

\newcommand{\Spec}{\op{Spec}}

\newcommand{\GL}{\op{GL}}

\newcommand{\Orth}{\op{O}}
\newcommand{\SO}{\op{SO}}

\newcommand{\Gal}{\op{Gal}}
\newcommand{\sep}[1]{#1_{\op{sep}}}
\newcommand{\alg}[1]{#1_{\op{alg}}}
\newcommand{\pcl}[1]{#1^{(p)}}
\newcommand{\Aut}{\op{Aut}}

\newcommand{\Gm}{\op{\mathbb{G}}_m}
\newcommand{\Hom}{\op{Hom}}

\newcommand{\Ex}{\op{Ex}}
\newcommand{\Ext}{\widetilde{\op{Ex}}}
\newcommand{\ZZ}{\mathbb{Z}}
\newcommand{\GG}{\mathbb{G}}
\newcommand{\FF}{\mathbb{F}}
\newcommand{\Fields}{\op{Fields}}
\newcommand{\Sets}{\op{Sets}}

\newcommand{\R}{\op{R}}

\newcommand{\Stab}{\op{Stab}}

\newcommand{\gap}{\op{gap}}
\makeatletter
\DeclareRobustCommand*{\bfseries}{%
   \not@math@alphabet\bfseries\mathbf
   \fontseries\bfdefault\selectfont
   \boldmath
 }
\makeatother

\usepackage{hyperref}

\begin{document}
\title[Essential $p$-dimension]{Essential $p$-dimension 
of algebraic groups whose connected component is a torus}

\author[R. L\"otscher]{Roland L\"otscher$^{(1)}$}
\address{Mathematisches Institut, Universit\"at M\"unchen, 80333
M\"unchen, Germany}
\email{roland.loetscher@mathematik.uni-muenchen.de}
\thanks{$^{(1)}$ Supported by the Deutsche Forschungsgemeinschaft, GI 706/2-1.}

\author[M. M{\tiny{ac}}Donald]{Mark M{\tiny{ac}}Donald$^{(2)}$}
\address{Department of Mathematics, University of British Columbia,
Vancouver V6T1Z2, Canada}
\email{mlm@math.ubc.ca}
\thanks{$^{(2)}$ Partially supported by a postdoctoral fellowship
from the Canadian National Science and Engineering 
Research Council (NSERC)}

\author[A. Meyer]{Aurel Meyer$^{(3)}$}
\address{D\'epartment de Math\'ematiques, Universit\'e Paris-Sud 11,
91405 Orsay, France}
\email{aurel.meyer@math.u-psud.fr}
\thanks{$^{(3)}$ Supported by a postdoctoral fellowship
from the Swiss National Science Foundation}

\author[Z. Reichstein]{Zinovy Reichstein$^{(4)}$}
\address{Department of Mathematics, University of British Columbia,
Vancouver V6T1Z2, Canada}
\email{reichst@math.ubc.ca}
\thanks{$^{(4)}$ Partially supported by a Discovery grant 
from the Canadian National Science and Engineering 
Research Council (NSERC)}

\subjclass[2000]{20G15, 11E72}

% 20G05 Representation theory (linear algebraic groups)
% 20G15 Linear algebraic groups over arbitrary fields
% 11E72 (1985-now)  Galois cohomology of linear algebraic groups
% 16K  (1991-now)  Division rings and semisimple Artin rings
% 16K20   (1991-now)  Finite-dimensional
% 20G15 (1973-now)  Linear algebraic groups over arbitrary fields
% 20D15 (1973-now)  Nilpotent groups, $p$-groups
% 20C10  (1973-now)  Integral representations of finite groups

\keywords{Essential dimension, algebraic torus, twisted finite group, generically free representation}

\begin{abstract}
Following up on our earlier work and 
the work of N.~Karpenko and A.~Merkurjev, we study the 
essential $p$-dimension of linear algebraic groups $G$ whose 
connected component $G^0$ is a torus.
\end{abstract}

\maketitle
\tableofcontents

\section{Introduction}
\label{sec:introduction} 
Let $p$ be a prime integer and $k$ a base field of characteristic $\neq p$. 
In this paper we will study the essential dimension
of linear algebraic $k$-groups $G$ whose connected component $G^0$ is 
an algebraic torus.  This is a natural class of groups; for 
example, normalizers of maximal tori in reductive linear 
algebraic groups are of this form.
We will use the notational conventions of \cite{LMMR}.
For background material and further references on the notion 
of essential dimension, see~\cite{Re2}.

For the purpose of computing $\ed(G; p)$ we may replace 
the base field $k$ by a $p$-special closure $\pcl{k}$.
A $p$-special closure of $k$ is a maximal directed limit 
of finite prime to $p$ extensions
over $k$; for details see Section~\ref{sec:primeToPClosure}. The 
resulting field $k=\pcl{k}$ is then \emph{$p$-special}, i.e.~every 
field extension of $k$ has $p$-power degree. (Some authors use the 
term ``$p$-closed" in place of ``p-special".)
Furthermore, the finite group $G/G^0$ has a Sylow $p$-subgroup
$F$ defined over $k=\pcl{k}$; see~\cite[Remark 7.2]{LMMR}.
Since $G$ is smooth we may replace $G$ by the preimage of $F$ 
without changing the essential $p$-dimension; see~\cite[Lemma 4.1]{LMMR}. 
It is thus natural to restrict our attention to 
the case where $G/G^0$ is a finite $p$-group, 
i.e., to those $G$ which fit into an exact sequence
of $k$-groups of the form
\begin{equation} \label{e.basic-sequence}
 1 \to T \to G  \xrightarrow{\; \pi \;} F \to  1 \, ,
\end{equation}
where $T$ is a torus and $F$ is a smooth finite
$p$-group.  Note that $F$ may be twisted (i.e. non-constant) 
and $T$ may be non-split over $k$.  Moreover,
the extension~\eqref{e.basic-sequence} 
is not assumed to be split either.

To state our main result, recall that a linear representation
$\rho \colon G \to \GL(V)$ is called {\em generically free} if there exists
a $G$-invariant dense open subset $U \subset V$ such that
the scheme-theoretic stabilizer of every point of $U$ is trivial.
A generically free representation is clearly faithful but 
the converse does not always hold; see below.
% A key feature of the class of groups we are considering is that
% for these groups is that the converse
We will say that $\rho$ is $p$\textit{-generically free} 
(respectively, $p$\textit{-faithful}) if
$\ker \rho$ is finite of order prime to $p$, and $\rho$ descends
to a generically free (respectively, faithful) representation of $G/\ker\rho$. 

\begin{thm} \label{thm.MainTheorem}
Let $G$ be an extension of a (possibly twisted)
finite $p$-group $F$ by an algebraic torus $T$
defined over a $p$-special field $k$ of characteristic $\neq p$. Then
\[
\min \dim \rho - \dim G \le \ed(G;p) \le \min \dim \mu - \dim G \, ,
\]
where the minima are taken over all $p$-faithful linear
representations $\rho$ of $G$ and $p$-generically free
representations $\mu$ of $G$, respectively.
\end{thm}

If $G = G^0$ is a torus or $G = F$ is a finite $p$-group then 
the upper and lower bounds of Theorem~\ref{thm.MainTheorem} coincide
(see,~\cite[Lemma 2.5]{LMMR} and ~\cite[Remark 2.1]{mr1}, 
respectively). In these cases Theorem~\ref{thm.MainTheorem}
reduces to~\cite[Theorems 1.1 and 7.1]{LMMR}, respectively.
In the case of constant finite $p$-groups this result is due 
to N.~Karpenko and A.~Merkurjev, whose work~\cite{KM} was 
the starting point for both~\cite{LMMR} and the present paper.
We will show that the upper and lower bounds 
of Theorem~\ref{thm.MainTheorem} coincide for 
a larger class of groups, which we call {\em tame};
see Definition~\ref{def.tame} and Theorem~\ref{thm.tame}(b) below.  

In general, groups $G$ of the form~\eqref{e.basic-sequence} 
may have faithful (respectively, $p$-faithful) 
representations which are not generically free (respectively,
$p$-generically free). This phenomenon is not well understood;
there is no classification of such representations,
and we do not even know for which groups $G$ they 
occur 
\footnote{More is known about faithful representations which
are not generically free in the case where $G$ is connected semisimple.
For an overview of this topic, see~\cite[Section 7]{popov-vinberg}.}.
It is, however, the source of many of the subtleties 
we will encounter.

To give the reader a better feel for this phenomenon, let us briefly
consider the following ``toy" example. Let $k = \mathbb C$, $p = 2$ 
and $G = \Orth_2$, the group of $2 \times 2$ orthogonal matrices.
It is well known that $G \simeq \mathbb \GG_m \rtimes \mathbb Z/ 2 \mathbb Z$,
where $G^0 = \SO_2 = \mathbb \GG_m$ is a 1-dimensional torus.
It is easy to see that the natural representation 
$i \colon G \hookrightarrow \GL_2$ is the 
unique $2$-dimensional faithful representation of $G$.
However this representation is not generically free: the stabilizer
$\Stab_G(v)$ of every anisotropic vector $v = (a, b) \in \mathbb C^2$
is a subgroup of $G = \Orth_2$ of order $2$ generated by the reflection in the line 
spanned by $v$. Here ``anisotropic" means $a^2 + b^2 \neq 0$;
anisotropic vectors are clearly Zariski dense in $\mathbb C^2$.
On the other hand, the 3-dimensional 
representation $i \oplus \det$ is easily seen to be 
generically free. Here $\det$ 
denotes the one-dimensional representation $\det \colon \Orth_2 \to \GL_1$.  
Since $\dim G = 1$,
Theorem~\ref{thm.MainTheorem} yields $1 \le \ed(\Orth_2; 2) \le 2$. 
The true value of $\ed(\Orth_2; 2)$ is $2$; see~\cite[Theorem 10.3]{Re1}.

We now proceed to state our result about the gap between the upper
and lower bounds of Theorem~\ref{thm.MainTheorem}.
The group $C(F)$ which appears in the definition below is 
the maximal split $p$-torsion subgroup of the center of $F$; for 
a precise definition, see Section~\ref{sect.proof-main-thm}.

\begin{defn} \label{def.tame}
Let $G$, $T \colonequals G^0$ and $\pi \colon G \to F \colonequals G/T$ 
be as in~\eqref{e.basic-sequence}. 
Consider the natural (conjugation) action of $F$ on $T$.
We say that $G$ is {\em tame} if $C(F)$ lies in the kernel of this action. 
Equivalently, $G$ is tame if $T$ is central in $\pi^{-1}(C(F))$.
\end{defn}

For any group $G$ as in~\eqref{e.basic-sequence} over a $p$-special field $k$, we define $\gap(G;p)$ to be the difference between the dimensions of a minimal $p$-faithful $G$-representation, and a minimal $p$-generically free $G$-representation. This is precisely the ``gap'' between the upper and lower bounds in 
Theorem~\ref{thm.MainTheorem}.

\begin{thm} \label{thm.tame}
Let $k$ be a $p$-special field of characteristic 
$\neq p$ and $G$ be as in~\eqref{e.basic-sequence}. Then

\smallskip
(a) $\gap(G;p) \leq \dim \, T - \dim \, T^{C(F)}$.

\smallskip
(b) If $G$ is tame then $\gap(G;p)=0$, 
i.e.  $$\ed(G;p)=\min \dim \rho - \dim G,$$ where the minimum 
is taken over all $p$-faithful $k$-representations of $G$.
\end{thm}

In many cases the lower bound of Theorem~\ref{thm.MainTheorem} is 
much larger than $\dim T$, so Theorem~\ref{thm.tame}(a) may be 
interpreted as saying that the gap between the lower and upper bounds
of Theorem~\ref{thm.MainTheorem} is not too wide, even if $G$ is not tame.

As a consequence of Theorem~\ref{thm.MainTheorem}, 
we will also prove the following ``Additivity Theorem".

\begin{thm} \label{thm.additive}
Let $k$ a $p$-special field of characteristic $\neq p$ and $G_1$, $G_2$ be 
groups such that $\gap(G_1;p)= \gap(G_2;p)= 0$. 
Then $\gap(G_1 \times G_2;p)=0$ and
$\ed(G_1 \times G_2; p) = \ed(G_1; p) + \ed(G_2; p)$.
\end{thm}

The rest of this paper is structured as follows. 
In Section~\ref{sec:primeToPClosure} we discuss the notion of
$p$-special closure $\pcl{k}$ of a field $k$ and show that
passing from $k$ to $\pcl{k}$ does not change 
the essential $p$-dimension of any $k$-group. 
In Section~\ref{sect.isogeny} we show that if $G \to Q$ is
an isogeny of degree prime to $p$ then the essential $p$-dimensions
of $G$ and $Q$ coincide. Sections~\ref{sect.proof-main-thm},  
\ref{sect.DimIrredReps} and \ref{sect.faithful}
are devoted to the proof of our Main Theorem~\ref{thm.MainTheorem}.
In Sections~\ref{sect.tame} and~\ref{sect.gap} we prove 
Theorem~\ref{thm.tame} and in Section~\ref{sect.additive}  
we prove the Additivity Theorem~\ref{thm.additive}.
% In Section~\ref{sect.gap} we show that the gap between the 
% lower and upper bounds of Theorem~\ref{thm.MainTheorem} is not too wide.
In Section \ref{sect.edsmall} we classify central extensions 
$G$ of $p$-groups by tori of small essential $p$-dimension.
% In particular we give a precise description of those $G$ 
% for which $\ed(G;p) \leq p-2$.  

\section{The \texorpdfstring{$p$}{p}-special closure of a field}
\label{sec:primeToPClosure}

A field $L$ is called {\em $p$-special} if every finite extension of
$L$ has degree a power of $p$.
By \cite[Proposition 101.16]{EKM} there exists for every field $K$ an
algebraic field extension $L/K$ such that $L$ is $p$-special and every
finite sub-extension of $L/F$ has degree prime to $p$. Such a field
$L$ is called a $p$-special closure of $K$ and will be denoted by
$\pcl{K}$.
The following properties of $p$-special closures will be important for
us in the sequel.

\begin{lem}\label{ptoplem}
Let $K$ be a field with $\Char K \neq p$ and let $\alg{K}$ be an
algebraic closure of $K$ containing $\pcl{K}$.
\begin{enumerate}[label=(\alph*), ref=(\alph*)]
\item\label{ptop1} $\pcl{K}$ is a direct limit of finite
extensions $K_i/K$ of degree prime to $p$.
\item\label{ptop2} The field $\pcl{K}$ is perfect.
\item\label{ptop3} The cohomological $q$-dimension
of $\Psi=\Gal(\alg{K}/\pcl{K})$ is ${\rm cd}_q(\Psi)=0$ for
any prime $q\ne p$.
\end{enumerate}
\end{lem}
\begin{proof}
\ref{ptop1}The finite sub-extensions $K'/K$ of $\pcl{K}/K$ form a
direct system with limit $\pcl{K}$. Moreover the
degrees $[K':K]$ are all prime to $p$.
\ref{ptop2} Every finite extension of $\pcl{K}$ has $p$-power degree.
Since $\Char K\neq p$ it is separable.
\ref{ptop3} By construction $\Psi$ is a profinite $p$-group. The
result follows from \cite[Cor. 2, I. 3]{serre-gc}.
\end{proof}
We call a covariant functor $\mathcal{F} \colon
\Fields/k\to \Sets$ {\em limit-preserving}
if for any directed system of fields $\{K_i\}$,
$\displaystyle{\mathcal{F}(\lim_{\rightarrow}K_i)=
\lim_{\rightarrow}\mathcal{F}(K_i)}$.
For example if $G$ is an algebraic group, the functor $F_G = H^1(*,G)$
is limit-preserving; see \cite[2.1]{Ma}.

\begin{lem}
\label{lem.edPrimeToPClosure}
Let $\mathcal{F}$ be limit-preserving and $\alpha \in \mathcal{F}(K)$ an object.
Denote the image of $\alpha$ in $\mathcal{F}(\pcl{K})$ by $\alpha_{\pcl{K}}$.
\begin{enumerate}[label=(\alph*), ref=(\alph*)]
\item\label{edptop1}
$\ed_{\mathcal{F}}(\alpha;p)=\ed_{\mathcal{F}}(\alpha_{\pcl{K}};p)=
\ed_{\mathcal{F}}(\alpha_{\pcl{K}})$.
\item\label{edptop2}
$\ed(\mathcal{F};p)=\ed(\mathcal{F}_{\pcl{k}};p)$, where
$\mathcal{F}_{\pcl{k}} \colon \Fields/\pcl{k} \to \Sets$ denotes the
restriction of $\mathcal{F}$ to $\Fields/\pcl{k}$.
\end{enumerate}
\end{lem}

\begin{proof}
\ref{edptop1}
The inequalities
$\ed_{\mathcal{F}}(\alpha;p)\ge\ed_{\mathcal{F}}(\alpha_{\pcl{K}};p)=\ed_{\mathcal{F}}(\alpha_{\pcl{K}})$
are clear from the definition and since $\pcl{K}$ has no finite
extensions of degree prime to $p$.
It remains to prove
$\ed_{\mathcal{F}}(\alpha;p)\le\ed_{\mathcal{F}}(\alpha_{\pcl{K}})$.
If $L/K$ is finite of degree prime to $p$,
\begin{equation}\label{eq.ptp}
\ed_{\mathcal{F}}(\alpha;p)=\ed_{\mathcal{F}}(\alpha_L;p),
\end{equation}
cf. \cite[Proposition 1.5]{Me1} and its proof.
For the $p$-special closure $\pcl{K}$ this
is similar and uses \eqref{eq.ptp} repeatedly:

Suppose there is a subfield $K_0\subset \pcl{K}$ and
$\alpha_{\pcl{K}}$ comes from an element $\beta\in
\mathcal{F}(K_0)$, so that $\beta_{\pcl{K}}=\alpha_{\pcl{K}}$. Write
$\pcl{K}=\lim\mathcal{L}$, where $\mathcal{L}$ is a direct system of
finite prime to $p$ extensions of $K$. Then $K_0=\lim\mathcal{L}_0$
with $\mathcal{L}_0=\{L\cap K_0\mid L\in\mathcal{L}\}$ and by assumption
on $\mathcal{F}$,
$\displaystyle{\mathcal{F}(K_0)=\lim_{L^\prime\in\mathcal{L}_0}\mathcal{F}(L^\prime)}$.
Thus there is a field $L^\prime=L\cap K_0$ ($L\in \mathcal{L}$) and
$\gamma\in\mathcal{F}(L^\prime)$ such that $\gamma_{K_0}=\beta$.
Since $\alpha_L$ and $\gamma_L$ become equal over $\pcl{K}$, after
possibly passing to a finite extension, we may assume they are equal
over $L$ which is finite of degree prime to $p$ over $K$. Combining
these constructions with \eqref{eq.ptp} we see that
\[
\ed_{\mathcal{F}}(\alpha;p)=\ed_{\mathcal{F}}(\alpha_L;p)=\ed_{\mathcal{F}}(\gamma_L;p)\le\ed_{\mathcal{F}}(\gamma_L)\le\ed_{\mathcal{F}}(\alpha_{\pcl{K}}).
\]

\ref{edptop2} This follows immediately from \ref{edptop1}, taking
$\alpha$ of maximal essential $p$-dimension.
\end{proof}

\begin{prop}
\label{prop:bijectionPrimeToPClosure} Let
$\mathcal{F},\mathcal{G}\colon \Fields/k \to \Sets$ be
limit-preserving functors and $\mathcal{F} \to \mathcal{G}$ a
natural transformation. If the map
\[\mathcal{F}(K) \to \mathcal{G}(K)\]
is bijective (resp.~surjective) for any $p$-special field containing $k$
then
\[\ed(\mathcal{F};p)= \ed(\mathcal{G};p) \quad (\mbox{resp.
}\ed(\mathcal{F};p)\geq\ed(\mathcal{G};p)).\]
\end{prop}
\begin{proof}

Assume the maps are surjective. By Lemma~\ref{ptoplem}\ref{ptop1},
the natural transformation is $p$-surjective, in the terminology of
\cite{Me1}, so we can apply \cite[Prop.\ 1.5]{Me1} to conclude
$\ed(\mathcal{F};p)\geq\ed(\mathcal{G};p)$.

Now assume the maps are bijective. Let $\alpha$ be in
$\mathcal{F}(K)$ for some $K/k$ and $\beta$ its image in
$\mathcal{G}(K)$. We claim that
%\begin{equation}\label{eqel}
$\ed(\alpha; p)=\ed(\beta; p)$.
%\end{equation}
First, by Lemma~\ref{lem.edPrimeToPClosure} we can assume that $K$
is $p$-special and it is enough to prove that
$\ed(\alpha)=\ed(\beta)$.

Assume that $\beta$ comes from $\beta_0\in \mathcal{G}(K_0)$ for
some field $K_0\subset K$.
Any finite prime to $p$
extension of $K_0$ is isomorphic to a subfield
of $K$ (cf.~\cite[Lemma 6.1]{Me1})
and so also any $p$-special closure of $K_0$ (which has the
same transcendence degree over $k$).
We may therefore assume that $K_0$ is $p$-special.
By assumption $\mathcal{F}(K_0)\rightarrow
\mathcal{G}(K_0)$ and $\mathcal{F}(K)\rightarrow \mathcal{G}(K)$ are
bijective. The unique element $\alpha_0\in\mathcal{F}(K_0)$ which
maps to $\beta_0$ must therefore map to $\alpha$ under the natural
restriction map. This shows that $\ed(\alpha)\le\ed(\beta)$. The
other inequality always holds and the claim follows.

Taking $\alpha$ of maximal essential dimension, we obtain
$\ed(\mathcal{F};p)=\ed(\alpha;p)=\ed(\beta;p)\le\ed(\mathcal{G};p)$.
\end{proof}

\section{\texorpdfstring{$p$}{p}-isogenies}
\label{sect.isogeny}

An isogeny of algebraic groups is a surjective morphism $G\to Q$
with finite kernel. If the kernel is of order prime to
$p$ we say that the isogeny is a $p$-isogeny.
In this section we will prove
Proposition~\ref{prop.edQuotient} which says that $p$-isogenous
groups have the same essential $p$-dimension. 
% This result will play
% a key role in the proof of the upper bound in Theorem~\ref{thm.MainTheorem}.

\begin{prop}\label{prop.edQuotient}
Suppose $G \to Q$ is a $p$-isogeny of algebraic groups over a field $k$ of characteristic $\neq p$. Then
\begin{enumerate}[label=(\alph*), ref=(\alph*)]
\item\label{edquo1} For any $p$-special field $K$
containing $k$ the natural map
$H^1(K,G) \to H^1(K,Q)$ is bijective.
\item\label{edquo2} $\ed(G;p) = \ed(Q;p)$.
\end{enumerate}
\end{prop}

\begin{exa}
Let $E_6^{sc}, E_7^{sc}$ be simply connected simple groups
of type $E_6, E_7$ respectively.
In \cite[9.4, 9.6]{GR} it is shown that if
$k$ is an algebraically closed field of characteristic $\ne 2$ and
$3$ respectively, then
\[ \text{$\ed(E_6^{sc}; 2)=3$ and
$\ed(E_7^{sc}; 3)=3$.} \]
For the adjoint groups
$E_6^{ad}=E_6^{sc}/\mu_3$, $E_7^{ad}=E_7^{sc}/\mu_2$ we therefore have
% $\ed(E_6^{ad}; 2)=\ed(E_6^{sc}; 2)=3$ and
% $\ed(E_7^{ad}; 3)=\ed(E_7^{sc}; 3)=3$.
\[ \text{$\ed(E_6^{ad}; 2)=3$ and
$\ed(E_7^{ad}; 3)=3$.} \]
\end{exa}

For the proof of Proposition \ref{prop.edQuotient} will need two lemmas.

\begin{lem}\label{order}
\label{lem.equivalence} Let $N$ be a finite algebraic group over a
field $k$ of characteristic $\neq p$. The following are equivalent:
\begin{enumerate}[label=(\alph*), ref=(\alph*)]
\item\label{order1} $p$ does not divide the order of $N$.
\item\label{order2} $p$ does not divide the order of $N(\alg{k})$.
\end{enumerate}
If $N$ is also assumed to be abelian, denote by $N[p]$ the
$p$-torsion subgroup of $N$. The following are equivalent to the
above conditions.
\begin{enumerate}[label=(\alph*$'$), ref=(\alph*$'$)]
\item\label{order3} $N[p](\alg{k})=\{1\}$.
\item\label{order4} $N[p](\pcl{k})=\{1\}$.
\end{enumerate}
\end{lem}

\begin{proof}
\ref{order1}$\iff$\ref{order2}:\; Let $N^0$ be the connected
component of $N$ and $N^{et}=N/N^0$ the \'etale quotient. Recall
that the order of a finite algebraic group $N$ over $k$ is defined as
$|N|=\dim_k k[N]$ and $|N|=|N^0||N^{et}|$, see, for example,~\cite{Ta}. 
If $\Char k=0$, $N^0$ is trivial, if $\Char k=q\ne p$
is positive, $|N^0|$ is a power of $q$. Hence $N$ is of order
prime to $p$ if and only if the \'etale algebraic group $N^{et}$ is.
Since $N^0$ is connected and finite, $N^0(\alg{k})=\{1\}$ and
so $N(\alg{k})$ is of order prime to $p$ if and
only if the group $N^{et}(\alg{k})$ is. Then $|N^{et}|=\dim_k
k[N^{et}]=|N^{et}(\alg{k})|$, cf. \cite[V.29 Corollary]{Bou}.

\ref{order2}$\iff$\ref{order3} $\Rightarrow$ \ref{order4} are clear.

\ref{order3} $\Leftarrow$ \ref{order4}: Suppose $N[p](\alg{k})$ is
nontrivial. The Galois group $\Gamma=\Gal(\alg{k}/\pcl{k})$ is a
pro-$p$ group and acts on the $p$-group $N[p](\alg{k})$. The image
of $\Gamma$ in $\Aut(N[p](\alg{k}))$ is again a (finite) $p$-group
and the size of every $\Gamma$-orbit in $N[p](\alg{k})$ is a power
of $p$. Since $\Gamma$ fixes the identity in $N[p](\alg{k})$, this
is only possible if it also fixes at least $p-1$ more elements. It
follows that $N[p](\pcl{k})$ is non-trivial.
\end{proof}

\begin{rem}
Part \ref{order4} could be replaced by the slightly stronger
statement that $N[p](\pcl{k}\cap\sep{k})=\{1\}$, but we will not
need this in the sequel.
\end{rem}

% \begin{lem}\label{copgrp}
% Let $\Gamma$ be a profinite group, $G$ an (abstract) finite
% $\Gamma$-group and $|\Gamma|, |G|$ coprime.
% Then $H^1(\Gamma,G)=\{1\}$.
% \end{lem}
% 
% The case where $\Gamma$ is finite and $G$ abelian is classical. In
% the generality we stated, this lemma is also 
% known \cite[I.5, ex. 2]{serre-gc}.
%%%%%%%%%%%%%%%%%%%%%%%%%%%%%%%%%%%%%%%%%%%%%%%

\begin{proof}[{Proof of Proposition \ref{prop.edQuotient}}]
\ref{edquo1} Let $N$ be the kernel of $G\to Q$ and $K$ be a
$p$-special field over $k$.
Since $\sep{K}=\alg{K}$ (see Lemma~\ref{ptoplem}(b)), the sequence of
$\sep{K}$-points
$1\to N(\sep{K})\to G(\sep{K})\to Q(\sep{K})\to 1$ is exact.
By Lemma~\ref{order}, the order of $N(\sep{K})$ is not divisible by
$p$ and therefore coprime to the order of any finite quotient of
$\Psi=\Gal(\sep{K}/K)$.
By~\cite[I.5, ex. 2]{serre-gc} this implies that
 $H^1(K,N)=\{1\}$.
Similarly, if $_cN$ is the group $N$ twisted by a cocycle $c:\Psi\to
G$, $_cN(\sep{K})=N(\sep{K})$ is of order prime to $p$ and
$H^1(K,\,_cN)=\{1\}$.
It follows that $H^1(K,G)\to H^1(K,Q)$ is injective, cf. \cite[I.5.5]{serre-gc}.

Surjectivity is a consequence of \cite[I. Proposition 46]{serre-gc}
and the fact that the $q$-cohomological dimension of $\Psi$
is $0$ for any divisor $q$ of $|N(\sep{K})|$ (Lemma \ref{ptoplem}(c)).

This concludes the proof of part~\ref{edquo1}. Part~\ref{edquo2}
follows from part (a) and Proposition~\ref{prop:bijectionPrimeToPClosure}.
\end{proof}

\section{Proof of the Main Theorem: an overview}
\label{sect.proof-main-thm}

We will assume throughout this section that the field $k$ is $p$-special of $\Char k \ne p$, and $G$ is a smooth affine
$k$-group, such that $G^0 = T$ is a torus and $G/G^0 = F$ is a finite $p$-group,
as in~\eqref{e.basic-sequence}.  
%To simplify matters, we will replace $k$ by a
%$p$-special closure $\pcl{k}$ and thus will assume
%that $k = \pcl{k}$; by Lemma~\ref{lem.edPrimeToPClosure}
%this will not change $\ed(G; p)$. 

The upper bound in Theorem~\ref{thm.MainTheorem} is an easy consequence
of Proposition~\ref{prop.edQuotient}. Indeed, 
suppose $\mu \colon G \to \GL(V)$ is a $p$-generically free representation.
That is, $\Ker \mu$ is a finite group of order prime to $p$ and
$\mu$ descends to a generically free representation of
$G' \colonequals  G/\Ker \mu$. By Proposition~\ref{prop.edQuotient} 
$\ed(G;p)=\ed(G';p)$. On the other hand, 
\[ \ed(G'; p) \le \ed(G') \le \dim \mu - \dim G' = \dim \mu - \dim G ; \] 
see~\cite[Lemma 4.11]{BF} or~\cite[Corollary 4.2]{Me1}. 
This completes the proof of the upper bound in Theorem~\ref{thm.MainTheorem}.

The rest of this section will be devoted to outlining a proof 
of the lower bound of Theorem~\ref{thm.MainTheorem}. The details
(namely, the proofs of Propositions~\ref{prop3.2} and~\ref{prop.main2})
will be supplied in the next two sections.  The starting point of our argument 
is~\cite[Theorem 3.1]{LMMR}, which we reproduce 
as Theorem~\ref{thm:lowerbound} below for the reader's convenience.

\begin{thm} \label{thm:lowerbound}
Consider an exact sequence of algebraic groups over $k$
\[ 1 \to C \to G \to Q \to 1 \]
such that $C$ is central in $G$ and is isomorphic 
to $\mu_p^r$ for some $r \ge 0$.
Given a character $\chi \colon C \to \mu_p$ denote by $\Rep^{\chi}$ the class 
of irreducible representations $\phi \colon G \to \GL(V)$, 
such that $\phi(c) = \chi(c) \Id$ for every $c \in C$.

Assume further that 
\begin{equation} \label{e3.1} \gcd \{ \dim \phi \, | \, 
\phi \in \Rep^{\chi} \}
= \min \{ \dim \phi \, | \, \phi \in \Rep^{\chi} \}
\end{equation} 
for every character $\chi \colon C \to \mu_p$.  Then
\[ \ed(G;p) \geq \min \dim \psi - \dim G \, , \]
where the minimum is taken over all finite-dimensional 
representations $\psi$ of $G$
such that $\psi_{| \, C}$ is faithful.
\end{thm}

To prove the lower bound of Theorem~\ref{thm.MainTheorem},
we will apply Theorem~\ref{thm:lowerbound} to the exact sequence
\begin{equation} \label{e.soc}
1 \to C(G) \to G \to Q \to 1 \, .
\end{equation}
where $C(G)$ is a central subgroup of $G$ defined as follows.
Recall from \cite[Section 2]{LMMR} that
if $A$ is a $k$-group of multiplicative type, 
$\Split(A)$ is defined as the maximal split $k$-subgroup of $A$. 
That is, if $X(A)$ is the character $\Gal(\sep{k}/k)$-lattice of $A$ then
the character lattice of $\Split(A)$ is
defined as the largest quotient of $X(A)$ 
with trivial $\Gal(\sep{k}/k)$-action. 

If $G$ be an extension of a finite $p$-group by a torus, 
as in~\eqref{e.basic-sequence}, denote by $Z(G)[p]$ the $p$-torsion
subgroup of the center $Z(G)$. 
Note that $Z(G)$ is a commutative group, which is an extension
of a $p$-group by a torus. Since $\Char k \neq p$ it is smooth.
Moreover $Z(G)(\alg{k})$ consists of semi-simple elements. It follows
that $Z(G)$ is of multiplicative type.  We now define 
$C(G)\colonequals \Split(Z(G)[p])$.

In order to show that Theorem~\ref{thm:lowerbound} can be aplied to
the sequence~\eqref{e.soc}, we need to check that condition~\eqref{e3.1}
is satisfied. This is a consequence of the following proposition, which
will be proved in the next section. 

\begin{prop} \label{prop3.2}
Let \[1 \to T \to G \to F \to 1 \] be an exact sequence of (linear) 
algebraic $k$-groups, where $F$ is a finite $p$-group and $T$ a torus.
Then the dimension of every irreducible representation
of $G$ over $k$ is a power of $p$.
\end{prop}

Applying Theorem~\ref{thm:lowerbound} to the exact 
sequence~\eqref{e.soc} now yields
\[
\ed(G;p) \ge \min \; \dim \rho - \dim G \, ,
\]
where the minimum is taken over all representations
$\rho \colon G \to \GL(V)$ such that $\rho_{|\, C(G)}$ is faithful.
This resembles the lower bound of Theorem~\ref{thm.MainTheorem};
the only difference is that in the statement of  
Theorem~\ref{thm.MainTheorem} we take the minimum
over $p$-faithful representations $\rho$, and here we ask
that $\rho_{| \, C(G)}$ should be faithful.
The following proposition shows that the two bounds are, 
in fact, the same, thus completing the proof of Theorem~\ref{thm.MainTheorem}.   

\begin{prop} \label{prop.main2}
A finite-dimensional representation $\rho$ of $G$ is $p$-faithful
if and only if $\rho_{| \, C(G)}$ is faithful.
\end{prop}

We will prove Proposition~\ref{prop.main2} in Section~\ref{sect.faithful}.

\begin{rem} \label{rem.FiniteImpliesTorus}
The inequality 
$\min \dim \rho - \dim G \le \ed(G;p)$ 
of Theorem~\ref{thm.MainTheorem},
where $\rho$ ranges over all $p$-faithful representations
of $G$, fails if we take the minimum over just the faithful
(rather than $p$-faithful) representations, 
even in the case where $G = T$ is a torus.

Indeed, choose $T$ so that
the $\Gal(\sep{k}/k)$-character lattice $X(T)$ of $T$ is a direct 
summand of a permutation lattice, but $X(T)$ itself
is not permutation (see \cite[8A.]{CTS1} for 
an example of such a lattice). 

In other words, there exists a $k$-torus $T'$ such that
$T \times T'$ is quasi-split. This implies that
$H^1(K, T \times T') = \{1 \}$ and thus $H^1(K, T) = \{ 1 \}$
for any field extension $K/k$. Consequently $\ed(T; p) = 0$ for every prime $p$.

On the other hand, we claim that the dimension of the minimal 
faithful representation of $T$ is strictly bigger than $\dim T$. 
Assume the contrary. Then by \cite[Lemma 2.6]{LMMR} there
exists a surjective homomorphism $f \colon P \to X(T)$
of $\Gal(\sep{k}/k)$-lattices, where $P$ is permutation and 
$\rank P = \dim T$. This implies that $f$ has finite kernel 
and hence, is injective. We conclude that $f$ is an isomorphism and 
thus $X(T)$ is a permutation $\Gal(\sep{k}/k)$-lattice, a contradiction.
\qed
\end{rem}

\section{Dimensions of irreducible representations}
\label{sect.DimIrredReps}

The purpose of this section is to prove Proposition~\ref{prop3.2}.
Recall that $G^0$ is a torus, which we denote by $T$ and
$G/G^0$ is a finite $p$-group which we denote by $F$.
The case, where $T = \{ 1 \}$, i.e., $G = F$ is an arbitrary 
(possibly twisted) finite $p$-group, is established in 
the course of the proof of~\cite[Theorem~7.1]{LMMR}. 
Our proof of Proposition~\ref{prop3.2} below is based 
on leveraging this case as follows.

\begin{lem} \label{lem3.3}
Let $G$ be an algebraic group defined over a field $k$ and
\[ F_1 \subseteq F_2 \subseteq \dots \subset G \]
be an ascending sequence of finite $k$-subgroups whose union
$\cup_{n \ge 1} F_n$ is Zariski dense in $G$.
If $\rho \colon G \to \GL(V)$ is an irreducible representation
of $G$ then $\rho_{| \, F_i}$ is irreducible for
sufficiently large integers $i$.
\end{lem}

\begin{proof}
For each $d = 1, ..., \dim V - 1$ consider the $G$-action
on the Grassmannian $\Gr(d, V)$ of $d$-dimensional subspaces of $V$.
Let $X^{(d)} = \Gr(d, V)^G$ and
$X_i^{(d)} = \Gr(d, V)^{F_i}$
be the subvariety of $d$-dimensional $G$- (resp.~$F_i$-)invariant
subspaces of $V$.
Then $X_1^{(d)} \supseteq X_2^{(d)} \supseteq \ldots$ and
since the union of the groups $F_i$ is dense in $G$,
\[ X^{(d)} = \cap_{i \ge 0} X_i^{(d)} \, . \]
By the Noetherian property of $\Gr(d, V)$, we have
$X^{(d)} = X_{m_d}^{(d)}$ for some $m_d \ge 0$.

Since $V$ does not have any $G$-invariant $d$-dimensional
$k$-subspaces, we know that $X^{(d)}(k) = \emptyset$.
Thus, $X_{m_d}^{(d)} (k) = \emptyset$, i.e.,
$V$ does not have any $F_{m_d}$-invariant $d$-dimensional
$k$-subspaces.
Setting $m \colonequals \max \{ m_1, \dots, m_{\dim V - 1} \}$, we see that
$\rho_{| \, F_m}$ is irreducible.
\end{proof}

We now proceed with the proof of Proposition~\ref{prop3.2}.
By Lemma~\ref{lem3.3}, it suffices to
construct a sequence of finite $p$-subgroups
\[ F_1 \subseteq F_2 \subseteq \dots \subset G \]
defined over $k$ whose union $\cup_{n \ge 1} F_n$ is Zariski
dense in $G$.

In fact, it suffices to construct one $p$-subgroup
$F' \subset G$, defined over $k$ such that $F'$
surjects onto $F$. Indeed, once $F'$ is constructed,
we can define $F_i \subset G$ as
the subgroup generated  by $F'$ and $T[p^i]$, for every $i \ge 0$.
Since $\cup_{n \ge 1} F_n$ contains both $F'$ and
$T[p^i]$, for every $i \ge 0$ it is Zariski dense
in $G$, as desired.

The following lemma, which establishes the existence of $F'$,
is thus the final step in our proof of Proposition~\ref{prop3.2}.

\begin{lem} \label{lem.schneider}
Let $1 \to T \to G \xrightarrow{\pi} F \to 1$ be an extension
of a $p$-group $F$ by a torus $T$ over a field $k$, as in~\eqref{e.basic-sequence}.
Then $G$ has a $p$-subgroup $F'$ with $\pi(F')= F$.
\end{lem}

In the case where $F$ is split and $k$ is algebraically closed
this is proved in~\cite[p. 564]{cgr}; cf. also the proof of
\cite[Lemme 5.11]{bs}.

\begin{proof}
Denote by $\Ext^1(F,T)$ the group of equivalence
classes of extensions of $F$ by $T$.
We claim that $\Ext^1(F,T)$ is torsion.
Let $\Ex^1(F,T)\subset\Ext^1(F,T)$ be the classes
of extensions which have a scheme-theoretic
section (i.e. $G(K)\to F(K)$ is surjective for all $K/k$).
There is a natural isomorphism $\Ex^1(F,T)\simeq H^2(F,T)$, where
the latter one denotes Hochschild cohomology,
see~\cite[III. 6.2, Proposition]{DG}.
By~\cite{Sc2} the usual restriction-corestriction
arguments can be applied in Hochschild cohomology
and in particular, $m\cdot H^2(F,T)=0$ where $m$ is the order of $F$.
Now recall that $M\mapsto\Ext^i(F,M)$ and $M\mapsto\Ex^i(F,M)$
are both derived functors of the crossed homomorphisms
$M\mapsto\Ex^0(F,M)$, where in the first case $M$
is in the category of $F$-module sheaves and in
the second, $F$-module functors, cf. \cite[III. 6.2]{DG}.
Since $F$ is finite and $T$ an affine scheme,
by \cite[Satz 1.2 \& Satz 3.3]{Sc1} there is an exact sequence
of $F$-module schemes $1\to T \to M_1 \to M_2 \to 1$
and an exact sequence
$\Ex^0(F,M_1)\to\Ex^0(F,M_2)\to\Ext^1(F,T)\to H^2(F,M_1)\simeq\Ex^1(F,M_1)$.
The $F$-module sequence also induces a long exact
sequence on $\Ex(F,*)$ and we have a diagram
\[
\xymatrix{
&&\Ext^1(F,T)\ar@{->}[dr]&\\
\Ex^0(F,M_1)\ar@{->}[r]&\Ex^0(F,M_2)\ar@{->}[ur]\ar@{->}[dr]&&\Ex^1(F,M_1)\\
&&\Ex^1(F,T)\ar@{->}[ur]\ar@{^{(}->}[uu]&
}
\]
An element in $\Ext^1(F,T)$ can thus be killed first
in $\Ex^1(F,M_1)$ so it comes from $\Ex^0(F,M_2)$.
Then kill its image in $\Ex^1(F,T)\simeq H^2(F,T)$, so it comes from
$\Ex^0(F,M_1)$, hence is $0$ in $\Ext^1(F,T)$.
In particular, multiplying twice by the order $m$
of $F$, we see that $m^2\cdot\Ext^1(F,T)=0$.  This proves the claim.

Now let us consider the exact sequence
$1\to N\to T\xrightarrow{\times m^2} T\to 1$,
where $N$ is the kernel of multiplication by $m^2$.
Clearly $N$ is finite and we have an induced exact sequence
\[
\Ext^1(F,N)\to\Ext^1(F,T)\xrightarrow{\times m^2}\Ext^1(F,T)
\]
which shows that the given extension $G$ comes from an extension $F'$
of $F$ by $N$.
Then $G$ is the pushout of $F'$ by $N\to T$ and we can
identify $F'$ with a subgroup of $G$.
\end{proof}

\section{Proof of Proposition~\ref{prop.main2}}
\label{sect.faithful}

We will prove Proposition~\ref{prop.C(G)} below;
Proposition~\ref{prop.main2} is an immediate consequence, with
$N = \Ker\rho$. Recall that $G$ is an algebraic group over a $p$-special field $k$ of characteristic $\neq p$ 
such that $G^0$ is a torus, $G/G^0$ is a finite $p$-group, 
as in~\eqref{e.basic-sequence}, and $C(G)$ is the split central 
$p$-subgroup of $G$ defined in Section~\ref{sect.proof-main-thm}.

\begin{prop} \label{prop.C(G)}
Let $N$ be a normal $k$-subgroup of
$G$. Then the following conditions are equivalent:

\smallskip
(i) $N$ is finite of order prime to $p$,

\smallskip
(ii) $N \cap C(G) = \{ 1 \}$,

\smallskip
(iii) $N \cap Z(G)[p] = \{ 1 \}$,
\end{prop}

In particular, taking $N = G$, we see that $C(G) \neq \{ 1 \}$ if $G\neq\{1\}$.

\begin{proof}
%First note that $Z(G)$ is a commutative group, which is an extension
%of a $p$-group by a torus. Since $\Char k \neq p$ it is smooth.
%Moreover $Z(G)(\alg{k})$ consists of semi-simple elements. It follows
%that $Z(G)$ is of multiplicative type. \par
(i)$\Longrightarrow$ (ii) is obvious, since $C(G)$
is a $p$-group.

(ii) $\Longrightarrow$ (iii). Assume the contrary:
$A \colonequals N \cap Z(G)[p] \neq \{ 1 \}$.
By \cite[Section 2]{LMMR}
\[ \{ 1 \} \neq C(A) \subset N \cap C(G) \, , \]
contradicting (ii).

Our proof of the implication (iii) $\Longrightarrow$ (i)
will rely on the following  assertion:

\smallskip
{\bf Claim:}
Let $M$ be a non-trivial normal finite $p$-subgroup of $G$ such
that the commutator $(G^0,M)$ is trivial. Then $M \cap Z(G)[p] \neq \{1\}$.

\smallskip
To prove the claim, note that $M(\sep{k})$ is non-trivial and
the conjugation action of $G(\sep{k})$ on $M(\sep{k})$ factors through
an action
of the $p$-group $(G/G^0)(\sep{k})$. Thus each orbit has $p^n$ elements for
some $n \ge 0$; consequently, the number of fixed points is
divisible by $p$. The intersection $(M \cap Z(G))(\sep{k})$ is
precisely the fixed point set for this action; hence,
$M \cap Z(G)[p] \neq \{ 1 \}$. This proves the claim.

We now continue with the proof of the implication (iii) $\Longrightarrow$ (i).
For notational convenience, set $T \colonequals G^0$.
Assume that $N \triangleleft G$ and $N \cap Z(G)[p] = \{ 1 \}$.
Applying the claim to the normal subgroup
$M \colonequals (N \cap T)[p]$ of $G$, we see that
$(N\cap T)[p] = \{ 1 \}$, i.e., $N \cap T$ is a finite group of order
prime to $p$.  The exact sequence
\begin{equation} \label{extension2}
1 \to N\cap T \to N \to \overline{N} \to 1 \, ,
\end{equation}
where $\overline{N}$ is the image of $N$ in $G/T$,
shows that $N$ is finite. Now observe that
for every $r \ge 1$, the commutator $(N, T[p^r])$ is a $p$-subgroup
of $N \cap T$. Thus $(N, T[p^r]) = \{ 1 \}$ for every $r \ge 1$.
We claim that this implies $(N, T) = \{ 1 \}$.
If $N$ is smooth, this is straightforward;
see~\cite[Proposition 2.4, p. 59]{Bo}. If $N$ is not smooth,
note that the map $c \colon N \times T \to G$
sending $(n, t)$ to the commutator $n t n^{-1} t^{-1}$
descends to $\overline{c} \colon
\overline{N} \times T \to G$ (indeed, $N \cap T$ clearly commutes with $T$).
Since $|\overline{N}|$ is a power of $p$
and $\Char(k) \ne p$, $\overline{N}$ is smooth over $k$, and we can
pass to the separable closure $\sep{k}$ and apply
the usual Zariski density argument to show that
the image of $\overline{c}$ is trivial.

We thus conclude that $N \cap T$ is central in $N$.
Since $\gcd(|N\cap T|, \overline{N})=1$, by
\cite[Corollary 5.4]{Sch} the extension~\eqref{extension2}
splits, i.e., $N \simeq (N \cap T) \times \overline{N}$. This turns
$\overline{N}$ into a finite $p$-subgroup of $G$ with $(T,\overline{N})=\{1\}$. 
The claim implies that $\overline{N}$ is trivial. 
Hence $N=N\cap T$ is a finite group of order prime to $p$, as claimed.

This completes the proof of Proposition~\ref{prop.C(G)} and 
thus of Theorem~\ref{thm.MainTheorem}.
\end{proof}

\section{Faithful versus generically free}
\label{sect.tame}

% In this section we will show that a faithful (respectively, 
% $p$-faithful) action of a tame group on any irreducible variety is
% generically free (respectively, $p$-generically free); 
% see Proposition~\ref{prop.tame} below. Theorem~\ref{thm.tame}
% is an immediate consequence of this and Theorem~\ref{thm.MainTheorem}. 
%%%%%%%%%%%%%%%%%%%%%%%%%%%%%%%%%%%%%%%%%%%%%%% 
Throughout this section we will assume that
$k$ is a $p$-special field of characteristic $\ne p$ and
$G$ is a (smooth) algebraic $k$-group such that 
$T \colonequals G^0$ is a torus and $F \colonequals G/G^0$ is 
a finite $p$-group, as in~\eqref{e.basic-sequence}. As we have seen
in the introduction, some groups of this type have faithful linear
representations that are not generically free. In this section
we take a closer look at this phenomenon.

If $F'$ is a subgroup of $F$ then we will use the notation $G_{F'}$ 
to denote the subgroup $\pi^{-1}(F')$ of $G$. Here 
$\pi$ is the natural projection $G \to G/T = F$.

\begin{lem} \label{lem.gap}
Suppose $T$ is central in $G$. Then 

\smallskip
(a) $G$ has only finitely many $k$-subgroups
$S$ such that $S \cap T = \{ 1 \}$.

\smallskip
(b) Every faithful action of $G$ on a geometrically irreducible variety $X$ 
is generically free. 
\end{lem}

\begin{proof} After replacing $k$ by its algebraic closure $\alg{k}$
we may assume without loss of generality that $k$ is algebraically closed.

\smallskip
(a) Since $F$ has finitely many subgroups, it suffices to show that for every
subgroup $F_0 \subset F$, there are only finitely many $S \subset G$
such that $\pi(S) = F_0$ and $S \cap T = \{ 1 \}$.

After replacing $G$ by $G_{F_0}$, we may assume that $F_0 = F$.
In other words, we will show that $\pi$ has at most finitely many sections
$s \colon F \to G$. Fix one such section, $s_0 \colon F \to G$.
Denote the exponent of $F$ by $e$.
Suppose $s \colon F \to G$ is another section. Then for every 
$f \in F(k)$, we can write $s(f) = s_0(f) t$ for some $t \in T$.
Since $T$ is central in $G$, $t$ and $s_0(f)$ commute. Since
$s(f)^e = s_0(f)^e = 1$, we see that $t^e = 1$. In other words, 
$t \in T(k)$ is an $e$-torsion element,
and there are only finitely many $e$-torsion elements in $T$. 
We conclude that there are only finitely many choices of $s(f)$ 
for each $f \in F$. Hence, there are only finitely many 
sections $F \to G$, as claimed.

\smallskip
(b) %Assume $\phi$ is not generically free.
The restriction of the $G$-action on $X$ to $T$ is faithful and hence, generically free;
cf., e.g.,~\cite[Proposition 3.7(A)]{Lo}. 
Hence, there exists a dense open $T$-invariant subset $U \subset X$ such that  
$\Stab_T(u) = \{ 1 \}$ for all $u\in U$. In other words, if $S = \Stab_G(u)$ 
then $S \cap T = \{ 1 \}$. 
By part (a) $G$ has finitely many non-trivial subgroups $S$ with 
this property.  Denote them by $S_1, \dots, S_n$. 
Since $G$ acts faithfully, $X^{S_i}$ is a proper 
closed  subvariety of $X$ for any $i = 1, \dots, n$. Since $X$ is irreducible,
\[ U' =  U \setminus ( X^{S_1} \cup \dots \cup X^{S_n}) \]  
is a dense open $T$-invariant subset of $X$, and the stabilizer $\Stab_G(u)$
is trivial for every $u \in U'$.
Replacing $U'$ by the intersection of its (finitely many) $G(\alg{k})$-translates,  
we may assume that $U'$ is $G$-invariant.This shows that the $G$-action on $X$ is generically free.
\end{proof}

\begin{prop} \label{prop.tame} 
%Suppose $k$ is a $p$-special field of characteristic $\neq p$. \par
\smallskip
(a) A faithful action of $G$ on a geometrically irreducible variety $X$ 
is generically free if and only if the action  
of the subgroup $G_{C(F)} \subseteq G$ on $X$ is generically free.

\smallskip
(b) A $p$-faithful action of $G$ on a geometrically irreducible variety 
$X$ is $p$-generically free if and only if the action 
of the subgroup $G_{C(F)} \subseteq G$ on $X$ is
$p$-generically free.
\end{prop}

\begin{proof} 
(a) The (faithful) $T$-action on $X$ is necessarily generically free
cf.~\cite[Proposition 3.7(A)]{Lo}. Thus 
by \cite[Expos\'{e} V, Th\'{e}or\`{e}me 10.3.1]{gabriel} 
(or \cite[Theorem 4.7]{BF}) 
$X$ has a dense open $T$-invariant subvariety $U$ 
defined over $k$, which is the total space of a $T$-torsor,
$U \to Y \colonequals U/G$, where $Y$ is also smooth and
geometrically irreducible. Since $G/T$ is finite, after replacing
$U$ by the intersection of its (finitely many) $G(\alg{k})$-translates,  
we may assume that $U$ is $G$-invariant.

The $G$-action on $U$ descends to an $F$-action on $Y$ 
(by descent). Now it is easy to see that the following conditions 
are equivalent:

\smallskip
(i) the $G$-action on $X$ is generically free, and

\smallskip
(ii) the $F$-action on $Y$ is generically free;

\smallskip
\noindent
cf.~\cite[Lemma 2.1]{LR}. 
Since $F$ is finite, (ii) is equivalent to 

\smallskip
(iii) $F$ acts faithfully on $Y$.

\smallskip
\noindent
Since $k$ is $p$-special, Proposition \ref{prop.C(G)} tells us that the kernel 
of the $F$-action on $Y$ is trivial iff the kernel of the $C(F)$-action 
on $Y$ is trivial. In other words, (iii) is equivalent to 

\smallskip
(iv) $C(F)$ acts faithfully (or equivalently, generically freely) on $Y$ 

\smallskip
\noindent
and consequently, to 

\smallskip
(v) the $G_{C(F)}$-action on $U$ (or, equivalently, on $X$) is generically 
free.

\smallskip
\noindent
Note that (iv) and (v) are the same as (ii) and (i), respectively, 
except that $F$ is replaced by $C(F)$ and $G$ by $G_{C(F)}$.
Thus the equivalence of (iv) and (v) follows from the equivalence 
of (i) and (ii).
We conclude that (i) and (v) are equivalent, as desired.

\smallskip
(b) Let $K$ be the kernel of the $G$-action on $X$, which is contained 
in $T$ by assumption. Notice that $(G/K) / (T/K) = G/T = F$. So 
part (a) says the $G/K$-action on $X$ is generically free
if and only if the $G_{C(F)}/K$-action on $X$ is generically 
free, and part (b) follows. 
\end{proof}

 \begin{cor}
 \label{cor.tame}
Consider an action of a tame $k$-group $G$ 
(see Definition~\ref{def.tame}) on a geometrically irreducible 
$k$-variety $X$.
 
\smallskip
(a) If this action is faithful then it is
generically free.
 
\smallskip
(b) If this action is $p$-faithful then it is is $p$-generically free.
\end{cor}

\begin{proof} (a) Since $G$ is tame,
$T$ is central in $G_{C(F)}$. Hence, the $G_{C(F)}$-action
on $X$ is generically free by Lemma~\ref{lem.gap}(b). 
By Proposition~\ref{prop.tame}(a), the $G$-action on $X$ is generically free. 

(b) Let $K$ be the kernel of the action. Note that $G/K$ is also tame.
Now apply part (a) to $G/K$.
\end{proof}

\section{Proof of Theorem~\ref{thm.tame}}
\label{sect.gap}

In this section we will prove the following proposition, which
implies Theorem~\ref{thm.tame}(a). 
Theorem~\ref{thm.tame}(b) is an immediate consequence of part (a) 
and Theorem~\ref{thm.MainTheorem} (or alternatively, of
Corollary~\ref{cor.tame}(b) and Theorem~\ref{thm.MainTheorem}).

We continue to use the notational conventions and assumptions 
on $k$ and $G$ described at the beginning of Section \ref{sect.tame}.

\begin{prop} \label{prop.gap}
Let $\rho \colon G \to \GL(V)$ be a linear representation of $G$.

(a) If $\rho$ is faithful then
$G$ has a generically free representation of dimension
$\dim \rho + \dim T - \dim \, T^{C(F)}$.

\smallskip
(b) If $\rho$ is $p$-faithful, then
$G$ has a $p$-generically free representation of dimension
$\dim \rho + \dim T - \dim \, T^{C(F)}$.
\end{prop}

\begin{proof} (a) $T^{C(F)}$ is preserved by the conjugation action of $G$, 
so the adjoint representation of $G$ decomposes
as $\Lie(T) = \Lie(T^{C(F)}) \oplus W$ for some $G$-representation $W$. Since the $G$-action on $\Lie(T)$ factors through $F$, the existence of $W$ follows from Maschke's Theorem.  
Let $\mu$ be the $G$-representation on $V \oplus W$. 
Then $\dim \mu = \dim \rho  + \dim T - \dim \, T^{C(G)}$. 
It thus remains to show that $\mu$ is a generically free 
representation of $G$. 

Let $K$ be the kernel of the $G_{C(F)}$ action on $\Lie(T)$. We claim $T$ is central in $K$. The finite $p$-group $K/T$ acts 
on $T$ (by conjugation), and it fixes the identity. 
By construction $K/T$ acts trivially on the tangent space 
at the identity, which implies $K/T$ acts trivially on $T$, 
since the characteristic is $\neq p$; cf.~\cite[Proof of 4.1]{GR}. 
This proves the claim. 

By Lemma~\ref{lem.gap}, the $K$-action on $V$ is generically free. 
Now $G_{C(F)}$ acts trivially on $\Lie(T^{C(F)})$, so $G_{C(F)}/K$ acts faithfully on $W$. Since $G_{C(F)}/K$ is finite, this action is also generically free. Therefore $G_{C(F)}$ acts generically freely on $V \oplus W$  \cite[Lemma 3.2]{mr1}. Finally, by Proposition~\ref{prop.tame}(a), $G$ acts generically freely on $V \oplus W$, as desired.

\smallskip
(b) By our assumption
$\ker \rho  \subset T$. Set $\overline{T} \colonequals T/\ker\rho$. It is easy to see that $\dim \, T^{C(F)} \le 
\dim \, \overline{T}^{C(F)}$. Hence, by part (a), there exists
a generically free representation of $G /\ker\rho$ of  
dimension 
\[ \dim \overline{T} - \dim \, \overline{T}^{C(F)} \le
\dim T - \dim \, T^{C(F)} \, . \]
We may now view this representation as a $p$-generically 
free representation of $G$. This completes the proof of Theorem~\ref{thm.tame}.
\end{proof}

\begin{rem}
A similar argument shows that for any tame normal subgroup $H \subset G$ 
over a $p$-special field $k$, $\gap(G;p) \leq \ed(G/H ;p)$.
\end{rem}

\section{Additivity}
\label{sect.additive}

Our proof of the Additivity Theorem~\ref{thm.additive} relies on the following lemma. Let $G$ be an algebraic group defined over $k$ and $C$ be a
$k$-subgroup of $G$. Denote the minimal dimension of a
representation $\rho$ of $G$ such that $\rho_{|\, C}$ 
is faithful by $f(G, C)$. 

\begin{lem} \label{lem.additive}
For $i = 1, 2$ let $G_i$ be an algebraic group defined over a field $k$ and
$C_i$ be a central $k$-subgroup of $G_i$. Assume that $C_i$ is
isomorphic to $\mu_p^{r_i}$ over $k$ for some $r_1, r_2 \ge 0$. Then
\[ f(G_1 \times G_2; C_1 \times C_2) = f(G_1; C_1) + f(G_2; C_2) \,
. \]
\end{lem}

Our argument below is a variant of the proof of~\cite[Theorem 5.1]{KM},
where $G$ is assumed to be a (constant) finite $p$-group and $C = C(G)$
(recall that $C(G)$ is defined at the beginning of 
Section~\ref{sect.proof-main-thm}).

\begin{proof} For $i = 1, 2$ let
$\pi_i \colon G_1 \times G_2 \to G_i$ be the natural projection and
$\epsilon_i \colon G_i \to G_1 \times G_2$ be the natural inclusion.

If $\rho_i$ is a $d_i$-dimensional representation of $G_i$ whose
restriction to $C_i$ is faithful, then clearly $\rho_1 \circ \pi_1
\oplus \rho_2 \circ \pi_2$ is a $d_1 + d_2$-dimensional
representation of $G_1 \times G_2$ whose restriction to $C_1 \times
C_2$ is faithful.  This shows that
\[ f(G_1 \times G_1; C_1 \times C_2) \le f(G_1; C_1) + f(G_2; C_2) \, . \]
To prove the opposite inequality, let $\rho \colon G_1 \times G_2
\to \GL(V)$ be a representation such that $\rho_{| \, C_1 \times
C_2}$ is faithful, and of minimal dimension
\[ d = f(G_1 \times G_1; C_1 \times C_2) \]
with this property. Let $\rho_1,\rho_2,\dotsc,\rho_n$ denote the
irreducible decomposition factors in a decomposition series of
$\rho$. Since $C_1 \times C_2$ is central in $G_1 \times G_2$, each
$\rho_i$ restricts to a multiplicative character of $C_1 \times C_2$
which we will denote by $\chi_i$. Moreover since $C_1 \times
C_2\simeq \mu_p^{r_1+r_2}$ is linearly reductive $\rho_{| \, C_1
\times C_2}$ is a direct sum $\chi_1^{\oplus d_1} \oplus \dotsb
\oplus \chi_n^{\oplus d_n}$ where $d_i=\dim V_i$. It is easy to see
that the following conditions are equivalent:

\smallskip
(i) $\rho_{| \, C_1 \times C_2}$ is faithful,

\smallskip
(ii) $\chi_1, \dots, \chi_n$ generate $(C_1 \times C_2)^*$ as an
abelian group.

\smallskip
\noindent In particular we may assume that $\rho=\rho_1\oplus \dotsb
\oplus \rho_n$. Since $C_i$ is isomorphic to $\mu_p^{r_i}$, we will
think of $(C_1 \times C_2)^*$ as a $\FF_p$-vector space of dimension
$r_1$ + $r_2$. Since (i) $\Leftrightarrow$ (ii) above, we know that
$\chi_1, \dots, \chi_n$ span $(C_1 \times C_2)^*$. In fact, they
form a basis of $(C_1 \times C_2)^*$, i.e., $n = r_1 + r_2$. Indeed,
if they were not linearly independent we would be able to drop some
of the terms in the irreducible decomposition $\rho_1 \oplus \dots
\oplus \rho_n$, so that the restriction of the resulting
representation to $C_1 \times C_2$ would still be faithful,
contradicting the minimality of $\dim \rho$.

We claim that it is always possible to replace each $\rho_j$ by
$\rho_j'$, where $\rho_j'$ is either $\rho_j \circ \epsilon_1 \circ
\pi_1$ or $\rho_j \circ \epsilon_2 \circ \pi_2$ such that the
restriction of the resulting representation $\rho' = \rho_1' \oplus
\dots \oplus \rho_n'$ to $C_1 \times C_2$ remains faithful.  Since
$\dim \rho_i = \dim \rho_i'$, we see that $\dim \rho' =
\dim \rho$. Moreover, $\rho'$ will then be of the form $\alpha_1
\circ \pi_1 \oplus \alpha_2 \circ \pi_2$, where $\alpha_i$ is a
representation of $G_i$ whose restriction to $C_i$ is faithful.
Thus, if we can prove the above claim, we will have
\begin{align*}
f(G_1 \times G_1; C_1 \times C_2) & = \dim \rho = \dim \rho' =
\dim \alpha_1 + \dim\alpha_2  \\
 &\ge f(G_1, C_1) + f(G_2, C_2) \, ,
\end{align*}
as desired.

To prove the claim, we will define $\rho_j'$ recursively for $j = 1,
\dots, n$. Suppose $\rho_1', \dots, \rho_{j-1}'$ have already be
defined, so that the restriction of
\[ \rho_1' \oplus \dots \oplus \rho_{j-1}' \oplus \rho_j \dots \oplus \rho_n \]
to $C_1 \times C_2$ is faithful. For notational simplicity, we will
assume that $\rho_1 = \rho_1', \dots, \rho_{j-1} = \rho_{j-1}'$.
Note that
\[ \chi_j = (\chi_j \circ \epsilon_1 \circ \pi_1) +
(\chi_j \circ \epsilon_2 \circ \pi_2) \, . \] Since $\chi_1, \dots,
\chi_n$ form a basis $(C_1 \times C_2)^*$ as an $\FF_p$-vector
space, we see that (a) $\chi_j \circ \epsilon_1 \circ \pi_1$
or (b) $\chi_j \circ \epsilon_2 \circ \pi_2$  does not lie in
$\Span_{\FF_p}(\chi_1, \dots, \chi_{j-1}, \chi_{j+1}, \dots,
\chi_n)$. Set
\[ \rho_j' \colonequals \begin{cases}
\text{$\rho_j \circ \epsilon_1 \circ \pi_1$ in case (a), and} \\
\text{$ \rho_j \circ \epsilon_2 \circ \pi_2$, otherwise.}
\end{cases} \]
Using the equivalence of (i) and (ii) above, we see
that the restriction of
\[ \rho_1 \oplus \dots \oplus \rho_{j-1} \oplus
\rho_j' \oplus \rho_{j+1}, \dots \oplus \rho_n \] to $C$ is
faithful. This completes the proof of the claim and thus of
Lemma~\ref{lem.additive}.
\end{proof}

\begin{proof}[Proof of Theorem~\ref{thm.additive}]
Clearly $G_1\times G_2$ is again a $p$-group extended by a torus.

%By Lemma~\ref{lem.edPrimeToPClosure} we
%may pass to a $p$-special closure $\pcl{k}$.
Reall that $C(G)$ is defined as
the maximal split $p$-torsion subgroup of the center of $G$; 
see Section~\ref{sect.proof-main-thm}.
It follows from this definition that 
\[ C(G_1 \times G_2) = C(G_1) \times C(G_2) \, ; \]
cf.~\cite[Lemma~2.1]{LMMR}.  Lemma~\ref{lem.additive} and 
Proposition~\ref{prop.main2} show that the minimal dimension of a $p$-faithful
representation is $f(G,C(G))=f(G_1,C(G_1))+f(G_2,C(G_2))$, which is
the sum of the minimal dimensions of $p$-faithful representations of
$G_1$ and $G_2$. For $i\in\{1,2\}$, since $\gap(G_i;p)=0$, there
exists a $p$-generically free representation $\rho_i$ of $G_i$ of
dimension $f(G_i,C(G_i))$. The direct sum $\rho_1\oplus \rho_2$ is a
$p$-generically free representation of $G$ and its dimension is
$f(G,C(G))$. It follows that $\gap(G_1\times G_2;p)=0$.
By Theorem~\ref{thm.MainTheorem}
\[ \ed(G; p) = f(G, C(G)) - \dim G \, ; \]
cf. Proposition~\ref{prop.main2}, and similarly for $G_1$ and $G_2$. Hence
the equality $\ed(G;p)=\ed(G_1;p)+\ed(G_2;p)$ follows. This concludes
the proof.
\end{proof}

%\begin{rem}
%The formula $\ed(G_1\times G_2;p)=\ed(G_1;p)+\ed(G_2;p)$ also holds
%for algebraic tori over fields of characteristic $p$. This can be
%proven along the lines above using \cite[Theorem 1.1]{LMMR} in place
%of Theorem \ref{thm.MainTheorem}.
%\end{rem}

\begin{exa}
Let $T$ be a torus over a field $k$.  Suppose there exists
an element $\tau$ in the absolute Galois group $\Gal(\sep{k}/k)$
which acts on the character lattice $X(T)$ via multiplication by
$-1$.  Then $\ed(T; 2) \ge \dim T$.
\end{exa}
\begin{proof}
Let $n\colonequals \dim T$. Over the fixed field $K\colonequals
(\sep{k})^\tau$ the torus $T$ becomes isomorphic to a direct product
of $n$ copies of a non-split one-dimensional torus $T_1$. Using
\cite[Theorem 1.1]{LMMR} it is easy to see that $\ed(T_1;2)=1$. By
the Additivity Theorem~\ref{thm.additive} we conclude that 
$\ed(T;2)\geq \ed(T_K;2)=\ed((T_1)^{n};2)=n\ed(T_1;2) = \dim T$.
\end{proof}

%%%%%%%%%%%%%%%%%%%%%%%%%%%%%%%%%%%%%%%%%%%%%%%%%
% Since direct products of tori and $p$-groups over a field $k$ of
% characteristic $\ne p$ are contained in $\PTm(k;p)$ we know how to
% compute their essential $p$-dimension. The following result shows that
% we can even compute their absolute essential dimension over $p$-special
% fields.
% 
% \begin{cor} \label{cor.edDirectProduct}
% Let $G=T \times F$ for a torus $T$ and a $p$-group $F$ over a field
% $k$ of characteristic $\ne p$. Then
% \[ \ed(G_{\pcl{k}}) = \ed(G_{\pcl{k}};p)=\ed(G;p) = \min \dim \rho,\]
% where the minimum runs over all $p$-faithful representations $\rho$ of
% $G_{\pcl{k}}$.
% \end{cor}
% \begin{proof}
% The equality $\ed(G_{\pcl{k}};p)=\ed(G;p)$ follows from Lemma
% \ref{lem.edPrimeToPClosure}. So we may assume $k$ is $p$-special. In
% particular $k$ contains a primitive $p$th root of unity.
% By \cite[Theorem 1.1 and Theorem 7.1]{LMMR} we have $\ed(T;p)=\ed(T)$
% and $\ed(F;p)=\ed(F)$.
% Combining this with the Additivity Theorem~\ref{thm.additive}, we obtain
% \[ \ed(T\times F) \leq \ed(T) + \ed(F) = \ed(T;p) + \ed(F;p) =
% \ed(T\times F;p) \leq \ed(T \times F). \]
% \end{proof}
%%%%%%%%%%%%%%%%%%%%%%%%%%%%%%%%%%%%%%%%%%%%%%%%%

We conclude this section with an example
which shows, in particular, that the property $\gap(G;p)=0$ 
is not preserved under base field extensions.
Let $T$ be an algebraic torus over a field $p$-special field $k$ of characteristic 
 $\neq p$ and $F$ a non-trivial $p$-subgroup of the constant 
 group $S_n$. Then we can form the wreath product
 $$T \wr F \colonequals T^n \rtimes F,$$
 where $F$ acts on $T^n$ by permutations. 

\begin{exa} \label{ex.wreath}
$\gap(T\wr F;p)=0$ if and only if $\ed(T;p)>0$.  Moreover,
\[ \ed(T\wr F;p) =
\begin{cases}
\ed(T^n;p)=n\ed(T;p), & \text{if } \ed(T;p)>0, \\
\ed(F;p), & \text{otherwise.} 
\end{cases} \]
\end{exa}

\begin{proof} %We may assume that $k$ is $p$-special. 
Let $W$ be a $p$-faithful $T$-representation 
of minimal dimension. By~\cite[Theorems 1.1]{LMMR},
$\ed(T; p) = \dim W - \dim T$.

Then $W^{\oplus n}$ is naturally a $p$-faithful $T\wr F$-representation. 
Lemma \ref{lem.additive} and Proposition~\ref{prop.main2} 
applied to $T^n$ tell us that $W^{\oplus n}$ has minimal dimension 
among all $p$-faithful representations of $T \wr F$. 

Suppose $\ed(T;p)>0$, i.e.,  $\dim W>\dim T$. 
The group $F$ acts faithfully on the rational quotient 
$W^{\oplus n}/T^n=(W/T)^n$, since $\dim W/T = \dim W-\dim T>0$. 
It is easy to see that
the $T\wr F$-action on $W^{\oplus n}$ 
is $p$-generically free; cf. e.g,~\cite[Lemma 3.3]{mr1}.
In particular, $\gap(T\wr F;p)=0$ and 
\[ \ed(T\wr F;p)=\dim W^{\oplus n}-\dim (T\wr F) = n(\dim W - \dim T) 
= n\ed(T;p)=\ed(T^n;p) \, , \] 
where the last equality follows from Theorem~\ref{thm.additive}. 

Now assume that $\ed(T;p)=0$, i.e., $\dim W=\dim T$.  The group $T\wr F$ 
cannot have a $p$-generically free representation $V$ of dimension 
$\dim W^{\oplus n}=\dim T \wr F$, since $T^n$ would then have a dense 
orbit in $V$.  It follows that $\gap(T\wr F;p)>0$. In order to compute its essential 
$p$-dimension of $T\wr F$ we use the fact that 
the natural projection $T\wr F \to F$ has a section. 
Hence, the map $H^1(-,T\wr F)\to H^1(-,F)$ also has a section 
and is, consequently, a surjection. This implies 
that $\ed(T\wr F;p)\geq \ed(F;p)$. 
Let $W'$ be a faithful $F$-representation of dimension 
$\ed(F;p)$. The direct sum $W^{\oplus n}\oplus W'$ considered 
as a $T\wr F$ representation is $p$-generically free. 
So, $\ed(T\wr F;p)= \ed(F;p)$. 
\end{proof}

\section{{Groups of low essential \texorpdfstring{$p$}{p}-dimension}}
\label{sect.edsmall}

In \cite{LMMR} we have identified tori of essential dimension $0$ as
those tori, whose character lattice is invertible, i.e. a direct
summand of a permutation module, see \cite[Example 5.4]{LMMR}. 
The following lemma shows that among the algebraic groups $G$ 
studied in this paper, i.e. extensions of $p$-groups by tori, 
there are no other examples of groups of $\ed(G;p)=0$.
\begin{lem} \label{lem.ed=0}
Let $G$ be an algebraic group over a field $k$ such that $G/G^0$ is a $p$-group. If $\ed(G;p)=0$, then $G$ is connected.
\end{lem}

\begin{proof}
Assume the contrary: $F \colonequals G/G^0 \ne \{ 1 \}$.
%In order to obtain a contradiction we may pass to 
%the algebraic closure $\alg{k}$ of $k$ 
%(note that $\ed_k(G; p) \ge \ed_{\alg{k}}(G_{\alg{k}}; p)$)
%and thus assume that $k$ is algebraically closed.
Let $X$ be an irreducible $G$-torsor over some field $K/k$.
For example, we can construct $X$ as follows. Start with
a faithful linear representation $G \hookrightarrow \GL_n$ for some $n \ge 0$.
The natural projection $\GL_n \to \GL_n/G$ is a $G$-torsor. Pulling back
to the generic point $\Spec(K) \to \GL_n/G$, we obtain an irreducible 
$G$-torsor over $K$.

Now $X/G^0 \to \Spec(K)$ is an irreducible $F$-torsor.  
Since $F\neq \{ 1 \}$ is not connected, this torsor is non-split.
As $F$ is a $p$-group, $X/G^0$ remains non-split over every prime to $p$ extension $L/K$. It follows that the degree of every closed point of $X$ is divisible by $p$, hence $p$ is a torsion prime of $G$. Therefore $\ed(G;p)>0$ by \cite[Proposition 4.4]{Me1}.
This contradicts the assumption $\ed(G;p)=0$ and so $F$ 
must be trivial.
\end{proof}

For the remainder of this section
we will assume the base field $k$ is of characteristic $\neq p$.
 
\begin{prop} \label{prop.ed<=p-2}
Let $G$ be a central extension of a $p$-group $F$ by a torus $T$. If $\ed(G;p) \leq p-2$, then $G$ is of multiplicative type.
\end{prop}

\begin{proof}
Without loss of generality, assume $k=\alg{k}$. 
By Theorem \ref{thm.MainTheorem}, there is a $p$-faithful representation 
$V$ of $G$, with $\dim V \leq \dim T + p-2$. 

First consider the case when $V$ is faithful. By Nagata's 
theorem~\cite{nagata} $G$ is linearly reductive, hence we can 
write $V=\bigoplus_{i=1}^r V_i$ for some non-trivial 
irreducible $G$-representations $V_i$.  Since $T$ is central 
and diagonalizable it acts by a fixed character on $V_i$, for every $i$. 
Hence $r \geq \dim T$ by faithfulness of $V$. 
It follows that $1\leq \dim V_i \leq p-1$ for each $i$. 
But every irreducible $G$-representation has dimension a power of $p$ 
(Proposition~\ref{prop3.2}), so each $V_i$ is one-dimensional. In other words, $G$ is of multiplicative type.
%Then we have a decomposition as a $T$-representation, $V = V_1 \oplus \cdots \oplus V_r$, where each $V_i$ is the (non-trivial) weight space of a character $\chi_i$. 
%Since $T \subset G$ is central, the $V_i$'s are preserved by $G$. Moreover, each $V_i$ decomposes into the direct sum of irreducible $G$-representations. 
%To see this in positive characteristic, recall there is a finite subgroup $F' \subset G$, as in Lemma \ref{lem.schneider}, whose subrepresentations in $V_i$ are the same as those for $G$; and by Maschke's theorem, $F'$ is completely reducible.

%Since $V$ is faithful, $r \geq \dim T$. Therefore, $1 \leq \dim V_i \leq p-1$ for each $i$. 
%But every irreducible $G$-representation has dimension a power of $p$ (Proposition \ref{prop3.2}), so each $V_i$ decomposes into the direct sum of 1-dimensional $G$-representations. 
%In other words, $G$ is of multiplicative type.

Now consider the general case, where $V$ is only $p$-faithful, and let $K \subset G$ be the kernel of that representation. 
Then $G/K$ is of multiplicative type, so it embeds into a torus $T_1$. 
Since $T$ is central in $G$, a subgroup $F'$ as in Lemma \ref{lem.schneider} is normal, so let $T_2=G/F'$, which is also a torus. 
The kernel of the natural map $G \to T_1 \times T_2$ is contained in $K \cap F'$, which is trivial since $p$ does not divide the order of $K$. In other words, $G$ embeds into a torus, and hence is of multiplicative type.
\end{proof}

\begin{exa}
Proposition~\ref{prop.ed<=p-2} does not generalize to tame groups. For 
a counter-example, assume the $p$-special field $k$ contains a primitive $p^2$-root of unity, 
and consider the group $G=\Gm^p \rtimes \ZZ/p^2$, where a generator 
in $\ZZ/p^2$ acts by cyclically permuting the $p$ copies of $\Gm$. 
The group $G$ is tame, because $C(\ZZ/p^2)= \ZZ/p = \mu_p$ acts trivially 
on $\Gm^p$. On the other hand, $G$ is not abelian and hence, is 
not of multiplicative type. 

We claim that $\ed(G;p) = 1$ and hence, $\le p-2$ for every odd prime $p$.
There is a natural $p$-dimensional faithful representation $\rho$
of $G$; $\rho$ embeds $\Gm^p$ into $\GL_p$ diagonally, 
in the standard basis $e_1, \dots, e_p$,
and $\ZZ/p^2$ cyclically permutes $e_1, \dots, e_p$.
Taking the direct sum of $\rho$ 
with the $1$-dimensional representation 
$\chi \colon G \to  \ZZ/p^2 = \mu_{p^2} 
\hookrightarrow \Gm = \GL_1$, we obtain a faithful $p+1$-dimensional representation
$\rho \oplus \chi$ which is therefore generically free 
by Corollary~\ref{cor.tame} (this can also be verified directly). 
Hence, $\ed(G;p) \le (p+1) - \dim(G) = 1$. On the other hand, by
Lemma~\ref{lem.ed=0}, we see that $\ed(G;p) \ge 1$ and thus $= 1$, as claimed. 
 \qed 
\end{exa}

Let $\Gamma_p$ a finite $p$-group, and let $\phi: P \to X$ be 
a map of $\ZZ \Gamma_p$-modules. As in \cite{LMMR}, we will 
call $\phi$ a {\it $p$-presentation} if $P$ is permutation, 
and the cokernel is finite of order prime to $p$. We will denote 
by $I$ the augmentation ideal of $\ZZ[\Gamma_p]$, and 
by $\overline{X} \colonequals X/(pX+IX)$ the largest $p$-torsion 
quotient with trivial $\Gamma_p$-action. The induced map on quotient 
modules will be denoted $\bar{\phi} : \overline{P} \to \overline{X}$. 

In the sequel for $G$ a group of multiplicative type over $k$, 
the group $\Gamma_p$ in the definition of ``$p$-presentation'' is 
understood to be a Sylow $p$-subgroup of $\Gamma=\Gal(\ell/k)$, 
where $\ell/k$ is a Galois splitting field of $G$.

\begin{lem} \label{lem.p-pres}
Let $\phi:P \to X$ be a map of $\ZZ \Gamma_p$-modules. Then the cokernel 
of $\phi$ is finite of order prime to $p$ if and only if $\bar{\phi}$ 
is surjective. 
\end{lem}

\begin{proof}
This is shown in \cite[proof of Theorem 4.3]{Me2}, and 
from a different perspective in \cite[Lemma 2.2]{LMMR}.
\end{proof}

\begin{prop} \label{prop.ed<=r<=p-2}
Let $G$ be a central extension of a $p$-group $F$ by a torus $T$, and let $0 \leq r \leq p-2$. The following statements are equivalent.
\begin{enumerate}[label=(\alph*), ref=(\alph*)]
\item $\ed(G;p) \leq r$ 
\item $G$ is of multiplicative type and there is a $p$-presentation $P \to X(G)$ whose kernel is isomorphic to the trivial $\ZZ \Gamma_p$-module $\ZZ^r$. 
\end{enumerate} 
\end{prop}

\begin{proof}
Assuming (a), by Proposition \ref{prop.ed<=p-2} $G$ is of multiplicative type. By \cite[Corollary 5.1]{LMMR} we know there is a $p$-presentation $P\to X(G)$, whose kernel $L$ is of free of rank $\ed(G;p)\leq p-2$. By \cite[Satz]{AP}, $\Gamma_p$ must act trivially on $L$. 

(b) $\Rightarrow$ (a) follows from \cite[Corollary 5.1]{LMMR}.
\end{proof}

%%%%%%%%%%%%%
%\begin{rem}
%\label{rem:p=2case}
%For $p=2$, part (a) of the above corollary is not true. The following is an example of a central extension $G$ of $2$-group by a torus with $\ed(G;2) = 1$, and which is not of multiplicative type. Let $V$ be the two-dimensional irreducible representation of $D_8$, the dihedral group of order 8. If %$\Gm$ acts on $V$ by scalar multiplication, then the kernel of the $\Gm \times D_8$ action on $V$ is %$\mu_2$, sitting diagonally inside $\Gm \times Z(D_8)$. So $G=(\Gm \times D_8) / \mu_2$, which is not of multiplicative type, acts faithfully on $V$. Since the torus is central inside $G$, by Lemma \ref{lem.gap}(b), this action is also generically free, and hence $\ed(G) \leq 1$. Also, by Lemma \ref%{lem.ed=0}, we have $\ed(G;2) \geq 1$.
%\end{rem}
%%%%%%%%%%%%

\begin{prop}\label{prop.ed=r}
Assume $G$ is multiplicative type, with a $p$-presentation $\phi \colon P \to X(G)$ whose kernel is isomorphic to the trivial $\ZZ \Gamma_p$-module $\ZZ^r$ for some $r\geq 0$. Then $\ed(G;p)\leq r$ and the following conditions are equivalent.

\begin{enumerate}[label=(\alph*), ref=(\alph*)]
\item $\ed(G;p) = r$,
%\item $\rank C(S) = \rank C(G)$ 
%\item $C(S) \to C(\Gm^r)$ is trivial
\item $\ker \phi$ is contained in $pP + IP$,
\item $\ker \phi$ is contained in 
$$\left\{\sum_{\lambda \in \Lambda} a_\lambda \lambda \in P \mid a_\lambda\equiv 0 \pmod p\  \forall \lambda \in \Lambda^{\Gamma_p} \right\}.$$
\end{enumerate}

Here $I$ denotes the augmentation ideal in $\ZZ [\Gamma_p]$ and $\Lambda$ is a $\Gamma_p$-invariant basis of $P$. 
%The group $C(*)$ is defined in Section \ref{sect.proof-main-thm}.
\end{prop}

\begin{proof}
$(a)\Leftrightarrow(b)$: We have a commutative diagram:
$$
\xymatrix{
1 \ar[r] &\ZZ^r \ar[r]\ar[d]& P \ar[r]^{\phi}\ar[d]\ar[rd]& X(G)\ar[d]\\
             &(\ZZ/p\ZZ)^r \ar[r]& \overline{P} \ar[r]^{\bar{\phi}}& \overline{X(G)} 
}
$$
with exact rows. By Lemma~\ref{lem.p-pres}, $\bar{\phi}$ is a surjection. Therefore $\ker \phi \subseteq pP+IP$ if and only if $\bar{\phi}$ is an isomorphism. \par
Write $P$ as a direct sum $P\simeq \bigoplus_{j=1}^m P_j$ of transitive permutation $\ZZ\Gamma_p$-modules $P_1,\dotsc,P_m$. Then $P/(pP+IP)\simeq \bigoplus_{j=1}^m P_j/(pP_j + IP_j)\simeq (\ZZ/p\ZZ)^m$.
If $\bar{\phi}$ is not an isomorphism we can replace $P$ by the direct sum $\hat{P}$ of only $m-1$ $P_j$'s without losing surjectivity of $\bar{\phi}$. The composition $\hat{P}\hookrightarrow P \to X(G)$ is then still a $p$-presentation of $X(G)$, by Lemma~\ref{lem.p-pres}. Hence $\ed(G;p)\leq \rank \hat{P}-\dim G < \rank P - \dim G = r$. \par
Conversely assume that $\bar{\phi}$ is an isomorphism. Let $\psi \colon P'\to X(G)$ be a $p$-presentation such that $\ed(G;p)=\rank \ker \psi$. Let $d$ be the index $[X(G):\phi(P)]$, which is finite and prime to $p$. Since the map $X(G)\to d\cdot X(G), x\mapsto dx$  is an isomorphism we may assume that the image of $\psi$ is contained in $\phi(P)$. We have an exact sequence $\Hom_{\ZZ\Gamma_p}(P',P)\to \Hom_{\ZZ\Gamma_p}(P',\phi(P))\to \mathrm{Ext}_{\ZZ\Gamma_p}^1(P',\ZZ^r)$ and the last group is zero by \cite[Lemma 2.5.1]{Lor}. 
Therefore $\psi=\phi \circ \psi'$ for some map $\psi'\colon P'\to P$ 
of $\ZZ\Gamma_p$-modules. Since $\bar{\phi}$ is an isomorphism 
and $\psi$ is a $p$-presentation it follows from Lemma~\ref{lem.p-pres} 
that $\psi'$ is a $p$-presentation as well, and in particular 
that $\rank P' \geq \rank P$. Thus 
$\ed(G;p)=\rank \ker \psi \geq \rank \ker \phi = r$.
\par \medskip
$(b)\Leftrightarrow(c)$: It suffices to show that $P^{\Gamma_p}\cap (pP+IP)$ consists precisely of the elements of $P^{\Gamma_p}$ of the form $\sum_{\lambda \in \Lambda} a_\lambda \lambda$ with $a_\lambda\equiv 0 \pmod p$ for all $\lambda \in \Lambda^{\Gamma_p}$, for any permutation $\ZZ\Gamma_p$-module $P$. One easily reduces to the case where $P$ is a transitive permutation module. Then $P^{\Gamma_p}$ consists precisely of the $\ZZ$-multiples of $\sum_{\lambda \in \Lambda} \lambda$ and $pP+IP$ are the elements $\sum_{\lambda \in \Lambda} a_\lambda \lambda$ with $\sum_{\lambda \in \Lambda}a_\lambda\equiv 0 \pmod p$. Thus for $n \in \ZZ$ the element $n\sum_{\lambda \in \Lambda}\lambda$ lies in $pP+IP$ iff $n\cdot |\Lambda|\equiv 0 \pmod p$, iff $n\equiv 0 \pmod p$ or $|\Lambda|\equiv 0 \pmod p$. Since $|\Lambda|$ is a power of $p$ the claim follows.
\end{proof}

\begin{exa}
Let $E$ be an \'etale algebra over a field $k$. It can be written as
$E=\ell_1\times\cdots\times \ell_m$ with some separable field extensions
$\ell_i/k$. The kernel of the norm map $n\colon \R_{E/k}(\Gm)\to \Gm$ is denoted by
$\R_{E/k}^{(1)}(\Gm)$. Let $G=n^{-1}(\mu_{p^r})$ for some $r\geq 0$. It is a group of multiplicative type fitting into an exact sequence $$1 \to \R_{E/k}^{(1)}(\Gm) \to G \to \mu_{p^r}\to 1.$$ Let $\ell$ be a finite Galois extension of $k$ containing $\ell_1,\dotsc,\ell_m$ (so $\ell$ splits $G$), $\Gamma=\Gal(\ell/k)$, $\Gamma_{\ell_i}=\Gal(\ell/\ell_i)$ and $\Gamma_p$ a $p$-Sylow subgroup of $\Gamma$. The character module of $G$ has a $p$-presentation
\[ P\colonequals \bigoplus_{i=1}^m\ZZ[\Gamma/\Gamma_{\ell_i}] \to X(G), \]
with kernel generated by the element $( p^r, \cdots , p^r) \in P$. This element is fixed by $\Gamma_p$, so $\ed(G;p) \leq 1$. It satisfies condition (c) of Proposition \ref{prop.ed=r} iff $r >0$ or every $\Gamma_p$-set $\Gamma/\Gamma_{\ell_i}$ is fixed-point free. Note that $\Gamma/\Gamma_{\ell_i}$ has $\Gamma_p$-fixed points if and only if $[\ell_i:k]=|\Gamma/\Gamma_{\ell_i}|$ is prime to $p$. 
We thus have
\[ \ed(G;p)=\left\{
\begin{array}{ll}
0, &\mbox{if $r=0$ and $[\ell_i:k]$ is prime to $p$ for some $i$,} \\
1, &\mbox{otherwise}.
\end{array}\right. \]
\end{exa}

%%%%%%%%%%%%%%%%%%%%%%%%%%%%%%%%%%%%%%%%%%%%%%%%%%%%

%\section*{Acknowledgments}

\end{document}